\documentclass[12pt]{article}
\usepackage{longtable}
\usepackage{textcomp}
\usepackage{amsmath}
\usepackage{amssymb}
\usepackage{mathrsfs}
\usepackage{amsfonts}
\usepackage{multirow}
\usepackage{cases}
\usepackage{graphicx}
\usepackage{tabularx}
\usepackage{color}
\usepackage{setspace}

\newtheorem{Theorem}{Theorem}%[section]
\newtheorem{Lemma}{Lemma}%[section]
\newtheorem{Example}{Example}%[section]
\newenvironment{proof}{{\textbf{Proof.}}\,}{\hfill$\hbox{\rule{5pt}{5pt}}$\\}

\begin{document}
\title{A mixed integer programming approach to the tensor complementarity problem\footnote{This work was supported by the National Nature Science
Foundation of China (Grant No. 11671220, 11771244) and the Nature Science Foundation of Shandong Province (ZR2016AM29).} }
\author{Shouqiang Du$^a$ \quad\quad Liping Zhang$^b$\thanks{Cooresponding Author. Email address: lzhang@math.tsinghua.edu.cn}\\
\footnotesize{$^a$ School of Mathematics and Statistic,
Qingdao University, Qingdao,  266071, China.}\\
\footnotesize{$^b$ Department of Mathematical Sciences,  Tsinghua University, Beijing, 100084, China.
}} \vskip6mm
 \date{}
\maketitle

{\bf Abstract:} In this paper, we establish a new approach to solve the tensor complementarity problem (TCP). A mixed integer programming model is given and the TCP is solved by solving the model. The TCP  is shown to be formulated as an equivalent mixed integer feasibility problem. Based on the reformulation, some conditions are obtained to guarantee the solvability  of the TCP. Specially, a sufficient condition is given for TCP without solutions. A necessary and sufficient condition is given for existence of solutions. We also give a concrete bound for the solution set of the TCP with positive definite tensors. Moreover, we show that the TCP with a diagonal positive definite tensor has a unique solution. Numerical experiments on several test problems illustrate the efficiency of the proposed approach in terms of the quality of the obtained solutions.


\vskip 6mm

\section{Introduction}

An $m$th-order $n$-dimensional tensor $\mathcal{A}=(a_{i_1\ldots i_m})$ is a multi-array of real entries
$a_{i_1\ldots i_m}$, where $i_j\in [n]:=\{1,\ldots,n\}$ for $j\in [m]:=\{1,\ldots,m\}$. Let $x\in \mathbb{R}^n$. Then $\mathcal{A}x^{m-1}$
is a vector in $\mathbb{R}^n$ with its $i$th component as
$$
(\mathcal{A}x^{m-1})_i:=\sum_{i_2,\ldots,i_m=1}^{n}a_{ii_2\ldots i_m}x_{i_2}\cdots x_{i_m}
$$
for $i\in[n]$. Obviously, each component of $\mathcal{A}x^{m-1}$ is a homogeneous polynomial of degree
$m-1$. For any $q\in \mathbb{R}^n$, we consider the tensor complementarity problem, a special class of
nonlinear complementarity problems, denoted by TCP$(\mathcal{A},q)$: finding  $x\in \mathbb{R}^n$ such that
$$
x\geq 0,\quad {\mathcal{A}}x^{m-1}+q\geq 0,\quad x^T({\mathcal{A}}x^{m-1}+q)=0,$$
or showing that no such vector exists. When $m=2$, TCP$(\mathcal{A},q)$ is just the well-known linear complementarity
problem (LCP). The notion of the tensor complementarity problem was used firstly by Song and Qi \cite{songqi2015}.
They showed TCP$(\mathcal{A},q)$ with a nonnegative tensor has a solution
if and only if all principal diagonal entries of such a tensor are positive. Recently,
many theoretical results about the properties of
the solution set of TCP$(\mathcal{A},q)$  have been developed, including existence of solution \cite{glqx,hsw10,songqi15,songqi2016,whb9},
global uniqueness of solution \cite{bhw,glqx}, boundedness of solution set \cite{cqw2016,dlq,sy2016,sq17,whb9}, stability
of solution \cite{ylh17}, sparsity of solution \cite{lqx2017}, solution methods \cite{llv17,xlx17}, and so on. In addition, an application of the
TCP was given in \cite{hq19}. Among them,
 Song and Qi \cite{songqi2016}
discussed the solution of TCP$(\mathcal{A},q)$  with a strictly semi-positive tensor. Che, Qi, Wei \cite{cqw2016} discussed the existence and uniqueness of solution of
TCP$(\mathcal{A},q)$  with some special tensors. Song and Yu \cite{sy2016} obtained global upper bounds of the solution of the TCP$(\mathcal{A},q)$
with a strictly semi-positive tensor. Luo, Qi and Xiu \cite{lqx2017} obtained the sparsest solutions
to TCP$(\mathcal{A},q)$ with a Z-tensor. Gowda, Luo, Qi and Xiu \cite{glqx} studied the various equivalent
conditions for the existence of solution to TCP$(\mathcal{A},q)$ with a Z-tensor. Ding, Luo and Qi \cite{dlq}
showed the properties of  TCP$(\mathcal{A},q)$ with a P-tensor. Bai, Huang and Wang \cite{bhw} considered
the global uniqueness and solvability for TCP$(\mathcal{A},q)$ with a strong P-tensor. Wang, Huang
and Bai \cite{whb9} gave the solvability of TCP$(\mathcal{A},q)$ with exceptionally regular tensors. Huang, Suo
and Wang \cite{hsw10} presented  several classes of Q-tensors. To the best of our knowledge, the study on solution algorithms for solving TCP$(\mathcal{A},q)$  has few results. So,  how to establish efficient solution methods for TCP$(\mathcal{A},q)$ is an interesting topic.

The tensor complementarity problem, as a
natural extension of the linear complementarity problem seems to have similar properties
to such a problem, and to have its particular and nice properties other than ones of the
classical nonlinear complementarity problem. So how to identify their good properties and
applications will be very interesting. In past several decades, there have been numerous
mathematical workers concerned with the solution of LCP and solution methods \cite{lcp,ncp}.
There are so many well-established and fruitful methods for solving LCP. Among them, solving LCP via integer programming is an interesting and important approach \cite{intp1,intp}. Motivated by the study on solving LCP via integer programming, we consider a new approach to solve TCP$(\mathcal{A},q)$ by a mixed integer programming.  We show that TCP$(\mathcal{A},q)$ is solvable if and only if  TCP$(\mathcal{A},\beta^{m-1}q)$ has a solution for  $\beta>0$. Based on the observation,  we first establish a mixed integer programming (MIP) model and want to get solutions by solving the MIP. Moreover, we show that every TCP$(\mathcal{A},q)$ can be written as an equivalent mixed integer feasibility problem. Since TCP$(\mathcal{A},q)$ with a positive definite tensor $\mathcal{A}$ has a nonempty and compact solution set \cite{cqw2016}, a concrete bound is given in this paper. Specially, we also prove that if $\mathcal{A}$ is a diagonal positive definite tensor then TCP$(\mathcal{A},q)$ has a unique solution, which answers Conjecture 5.1 in \cite{cqw2016}.

This paper is organized as follows. In Section 2,  we list some preliminaries. In Section 3, we establish a mixed integer programming approach to solve TCP$(\mathcal{A},q)$ and some properties of TCP$(\mathcal{A},q)$ are given. Some numerical results are reported in Section 4. Some conclusions are given in the final section.

\section{Preliminaries}

In this section, we define the notations and collect some basic definitions and facts,
which will be used later on.

Denote $\mathbb{R}^n:=\{(x_1,x_2,\ldots,x_n)^T: x_i\in \mathbb{R}, i\in[n]\}$,  where
$\mathbb{R}$ is the set of real numbers. Let $e:=(1,\ldots,1)^T\in \mathbb{R}^n$.  Denote the set of all real $m$th-order $n$-dimensional tensors by $T_{m,n}$. Then, $T_{m,n}$ is a linear space of dimension $n^m$.
Let $\mathcal{A}=(a_{i_1\ldots i_m})\in T_{m,n}$. If the entries $a_{i_1\dots i_m}$ are invariant under any permutation
of their indices, then $\mathcal {A}$ is called a symmetric tensor. Denote the set of all real $m$th-order $n$-dimensional symmetric tensors by $S_{m,n}$. Then, $S_{m,n}$ is a linear subspace of $T_{m,n}$. Let
$\mathcal{A}=(a_{i_1\ldots i_m})\in T_{m,n}$ and $x\in \mathbb{R}^n$. Then $\mathcal{A}x^m$ is a homogeneous polynomial of degree
$m$, defined by
$$
\mathcal{A}x^m:=\sum_{i_1,\ldots,i_m=1}^na_{i_1\ldots i_m}x_{i_1}\cdots x_{i_m}.
$$
Moreover, $\|x\|$ denotes $2$-norm of $x$ and $x^T$ represents the transpose of $x$. Given two column vectors $x,y\in \mathbb{R}^n$, $x^\top y$ represents the inner product of $x$ and $y$.

In the following, we recall the definitions of E-eigenvalues and Z-eigenvalues of tensors in $T_{m,n}$ in \cite{CPZ1,Li,Qi}.
Denote $\mathbb{C}^n:= \{(x_1,x_2,\ldots,x_n)^T: x_i\in \mathbb{C}, i\in[n]\}$, where $\mathbb{C}$ is the set of complex
numbers. For any vector $x\in \mathbb{C}^n$, $x^{[m-1]}$  is a vector in $\mathbb{C}^n$ with its $i$th component
defined as $x_i^{m-1}$ for $i\in [n]$. Let $\mathcal{A}\in T_{m,n}$.
If and only if there is a nonzero vector
$x\in \mathbb{C}^n$ and a number $\lambda \in \mathbb{C}$ such that
\begin{equation}\label{heigen}
 \mathcal{A}x^{m-1} = \lambda x, \quad x^\top x=1,
\end{equation}
then $\lambda$ is called an {\it E-eigenvalue} of $\mathcal{A}$ and $x$ is called an {\it E-eigenvector} of $\mathcal{A}$, associated
with $\lambda$. If the E-eigenvector $x$ is real, then the E-eigenvalue $\lambda$ is also real. In this case,
$\lambda$ and $x$ are called a {\it Z-eigenvalue} and a {\it Z-eigenvector} of $\mathcal{A}$, respectively. For a symmetric tensor, Z-eigenvalues always exist. An even order symmetric tensor is
positive (semi-)definite if and only if all of its Z-eigenvalues are
positive (non-negative) \cite{Qi}. By \cite[Theorem 5]{Qi}, for both the largest Z-eigenvalue $\lambda_{max}$ and the smallest
Z-eigenvalue $\lambda_{min}$ of a symmetric tensor $\mathcal {A}$, we have
\begin{equation}\label{largest}
\lambda_{\max}=\max\left\{\mathcal{A}x^m: x^\top x=1\right\}
\end{equation}
and
\begin{equation}\label{smallest}
\lambda_{\min}=\min\left\{\mathcal{A}x^m: x^\top x=1\right\}.
\end{equation}

\section{A mixed integer programming method}

In this section, we first propose a new approach to solve TCP$(\mathcal{A},q)$ by a mixed integer programming, and then we give a condition to guarantee the boundedness of its solution set.

Clearly, TCP$(\mathcal{A},q)$  is equivalent to the following constrained optimization problem with optimal value $0$,
\begin{eqnarray}\label{eq2}
\min && \phi(x):=x^\top(q+{\mathcal{A}}x^{m-1})\nonumber\\
s.t. && x\ge 0,\quad q+{\mathcal{A}}x^{m-1}\ge 0.
\end{eqnarray}
It is easy to see that $\phi(x)\geq 0$ if the feasible region of (\ref{eq2}) is nonempty, that is, the problem (\ref{eq2}) is bounded from below.

Based on the model (\ref{eq2}), some properties of the solution set of TCP$(\mathcal{A},q)$ are easily obtained in some special cases:
\begin{itemize}
\item Let $q\geq0$. It is clear that $x=0$ is a trivial solution of TCP$(\mathcal{A},q)$ for any tensor $\mathcal {A}\in T_{m,n}$. So, in general, we assume that $q\ngeq 0$.
\item Let $\mathcal {A}=(a_{i_1\ldots i_m})\in T_{m,n}$ be a diagonal tensor, i.e., its off-diagonal entries are zero. In this case, the constrained optimization problem (\ref{eq2}) is reduced as follows:
\begin{eqnarray*}
\min && \sum_{i=1}^{n}(a_{i\ldots i}x_{i}^{m-1}+q_i)x_i \\
s.t. &&   x_i\geq 0, \quad  a_{i\ldots i}x_{i}^{m-1}+q_i\geq0, \quad i\in [n].
\end{eqnarray*}
From the above model, we know that  TCP$(\mathcal{A},q)$ is infeasible if there exists $i_0\in [n] $ such that $a_{i_0\cdots i_0}\leq 0$ and $ q_{i_0}<0$. Moreover, if $ q_{i_0}<0$ and $a_{i_0\cdots i_0}>0$ then TCP$(\mathcal{A},q)$ has a solution $x^*$ with its $i$th component:
\begin{equation}\label{solution}
x^*_i=\left\{\begin{array}{ll}
                      0,& \mbox{if $q_i\ge 0$}. \\
                 \left(\frac{-q_i}{a_{i\cdots i}}\right)^{\frac{1}{m-1}},& \mbox{if $q_i<0$ and $a_{i\ldots i}>0$}.
                    \end{array}
                  \right.
                  \end{equation}
 \item In \cite[Conjecture 5.1]{cqw2016}, Che, Qi and Wei  proposed a conjecture that  if a diagonal tensor $\mathcal{A}$ is  positive definite  then  TCP$(\mathcal{A},q)$ has a unique solution. Here, we can give a confirmed answer based on (\ref{eq2}). Since the diagonal tensor $\mathcal{A}$ is positive definite, $a_{i\ldots i}>0$ for $i\in [n]$ and every solution $x$ of TCP$(\mathcal{A},q)$ satisfies
\begin{equation}\label{diag}
     x_i\ge 0,\quad a_{i\ldots i}x_{i}^{m-1}+q_i\geq0, (a_{i\ldots i}x_{i}^{m-1}+q_i)x_i=0, \quad i\in [n].
\end{equation}
When $q\ge 0$, in this case $x=0$. In fact, if there exists some $i^*\in [n]$ such that $x_{i^*}>0$, then $a_{i^*\ldots i^*}x_{i^*}^{m-1}+q_{i^*}>0$. This is  a  contradiction with (\ref{diag}). When $q<0$, in this case
$$
x=\left(\frac{-q_i}{a_{i\cdots i}}\right)^{\frac{1}{m-1}}>0.$$ In fact, if there exists some $i^*\in [n]$ such that $x_{i^*}=0$, then $a_{i^*\ldots i^*}x_{i^*}^{m-1}+q_{i^*}=q_{i^*}<0$. This is  a  contradiction with (\ref{diag}). In other cases, $x$ must be in the form of (\ref{solution}). Hence, TCP$(\mathcal{A},q)$ with $\mathcal{A}$ being diagonal and  positive definite has a unique solution.
\end{itemize}

We summarize the above discussions in the following theorem.
\begin{Theorem}\label{thm0}
Given a TCP$(\mathcal{A},q)$. If the vector $q\ge 0$ then TCP$(\mathcal{A},q)$ has a trivial solution. Let $\mathcal{A}$ be a diagonal tensor, if there exists some $i\in [n]$ such that $q_i<0$ and $a_{i\ldots i}\le 0$ then TCP$(\mathcal{A},q)$ has no solution. Otherwise, TCP$(\mathcal{A},q)$ has  a solution in the form of (\ref{solution}). Furthermore, if the diagonal tensor $\mathcal{A}$ is positive definite then TCP$(\mathcal{A},q)$ has  a unique solution $x^*\in \mathbb{R}^n$ with its $i$th component:
$$
x^*_i=\begin{cases}
0, &\mbox{if $q_i\ge 0$},\\
                 \left(\frac{-q_i}{a_{i\cdots i}}\right)^{\frac{1}{m-1}},& \mbox{if $q_i<0$},
                 \end{cases} \quad i\in [n].
                 $$
\end{Theorem}

Given a general TCP$(\mathcal{A},q)$ with $\mathcal{A}\in T_{m,n}$ and $q\in \mathbb{R}^n$, we now show that it can always be solved by solving a mixed integer programming. Moreover, every tensor complementarity problem can be written as an equivalent mixed integer feasibility problem. These results are based on the following lemma.
\begin{Lemma}\label{lem1}
 TCP$(\mathcal{A},q)$ has a solution if and only if  TCP$(\mathcal{A},\beta^{m-1}q)$ has a solution for all real number $\beta>0$.
\end{Lemma}
\begin{proof}
Let $x,y\in\mathbb{R}^n$ be the solutions of TCP$(\mathcal{A},q)$ and  TCP$(\mathcal{A},\beta^{m-1}q)$, respectively. Then $y=\beta x$.
\end{proof}

Consider  the following mixed integer programming (MIP):
\begin{eqnarray} \label{mip}
\max_{\alpha,y,z} && \alpha^{m-1} \\
s.t. && 0\leq\mathcal{A}y^{m-1}+\alpha^{m-1}q\leq e-z, \nonumber\\
&& 0\leq y\leq z,\quad \alpha \ge 0, \nonumber\\
&& z\in\{0,1\}^n. \nonumber
\end{eqnarray}
The following theorem shows the relationship between TCP$(\mathcal{A},q)$ and  MIP (\ref{mip}).
\begin{Theorem}\label{thm1}
Let $(\alpha^*,y^*,z^*)$ be any optimal solution of MIP (\ref{mip}). If $\alpha^*>0$,  then $x={y^*}/{\alpha^*}$ solves TCP$(\mathcal{A},q)$. If  $\alpha^*=0$, then TCP$(\mathcal{A},q)$ has no solution.
\end{Theorem}
\begin{proof} For any given $\mathcal{A}\in T_{m,n}$ and $q\in \mathbb{R}^n$ with $q\ngeq 0$, MIP (\ref{mip}) always has the feasible solution $\alpha=0$, $y=0$ and $z_i=0$ or $1$ for $i\in [n]$. For any feasible solution $(\alpha,y,z)$ of  MIP (\ref{mip}), it follows from the feasibility constraints that
\begin{equation}\label{alpha}
\alpha\le \left(\frac{1+\|\mathcal{A}\|_{\infty}}{\|q\|_{\infty}}\right)^{\frac1{m-1}}.
\end{equation}
Hence, MIP (\ref{mip}) is feasible and bounded. Suppose MIP (\ref{mip}) has an optimal solution $(\alpha^*,y^*,z^*)$ with $\alpha^*>0$. Let $x={y^*}/{\alpha^*}$, then we have $x\ge 0$ and
$$
(\alpha^*)^{m-1}(\mathcal{A}x^{m-1}+q)=\mathcal{A}(y^*)^{m-1}+(\alpha^*)^{m-1}q\ge 0\quad \Rightarrow \quad \mathcal{A}x^{m-1}+q\ge 0.$$
Furthermore, for each $i\in[n]$, either $x_i=0$ or $(\mathcal{A}x^{m-1})_i+q_i=0$, so that $x^\top(\mathcal{A}x^{m-1}+q)=0$. That is, $x$ solves TCP$(\mathcal{A},q)$.

Now suppose that the optimal solution $(\alpha^*,y^*,z^*)$ to MIP (\ref{mip}) gives $\alpha^*=0$. Then we will show that TCP$(\mathcal{A},q)$ has no solution. Proof by contradiction, we assume that TCP$(\mathcal{A},q)$ has a solution. By Lemma \ref{lem1}, TCP$(\mathcal{A},\beta^{m-1}q)$ has a solution. That is, for any $\alpha>0$, there exists $y\ge 0$ such that
$$
\mathcal{A}y^{m-1}+\alpha^{m-1}q\ge 0,\quad y^\top(\mathcal{A}y^{m-1}+\alpha^{m-1}q)=0.
$$
Thus implies that $(\alpha,y,z)$  with $z_i=0$ if $y_i=0$ and $z_i=1$ otherwise for $i\in [n]$, is a feasible solution of MIP (\ref{mip}). Hence, $\alpha^*\ge \alpha>0$. This is a contradiction and hence TCP$(\mathcal{A},q)$ has no solution.
\end{proof}

By Theorem \ref{thm1}, every feasible solution $(\alpha,y,z)$ with $\alpha>0$ of MIP (\ref{mip}) corresponds to a solution of TCP$(\mathcal{A},q)$. Therefore, solving  MIP (\ref{mip}), we may generate several solutions of TCP$(\mathcal{A},q)$.

We now show that every tensor complementarity problem can be written as an equivalent mixed integer feasibility problem. The next theorem gives sufficient conditions for TCP$(\mathcal{A},q)$ without solution.
\begin{Theorem}\label{thm2}
Given a general TCP$(\mathcal{A},q)$ with $\mathcal{A}\in T_{m,n}$ and $q\in \mathbb{R}^n$, then TCP$(\mathcal{A},q)$ has no solution if the system of inequalities
\begin{equation}\label{system}
0\le \mathcal{A}x^{m-1}+q \le \tau e-u,\quad 0\le x\le \tau^{\frac{m-2}{1-m}}u,\quad 0\le u\le \tau e,\quad \tau\ge 1,
\end{equation}
is infeasible.
\end{Theorem}
\begin{proof} We assume by contradiction that TCP$(\mathcal{A},q)$ has a solution. By Lemma \ref{lem1}, TCP$(\mathcal{A},\beta q)$ also has a solution for any $\beta>0$. By Theorem \ref{thm1}, MIP (\ref{mip}) has an optimal solution $(\alpha, y,z)$ with $\alpha>0$. Without loss of generality we may assume that $\alpha\le 1$ due to (\ref{alpha}). Hence,
\begin{eqnarray*}
&& 0\le \mathcal{A}y^{m-1}+\alpha^{m-1}q \le e-z,\\
&& 0\le y\le z,\quad z\in\{0,1\}^n,\\
&& 0<\alpha \le 1.
\end{eqnarray*}
In the above system, we set
$$
\tau=\frac1{\alpha^{m-1}},\quad x=\frac{y}{\alpha},\quad u=\frac{z}{\alpha^{m-1}}.
$$
Then $(\tau, x,u)$ is a solution to the following system
$$
0\le \mathcal{A}x^{m-1}+q \le \tau e-u,\quad 0\le x\le \tau^{\frac{m-2}{1-m}}u,\quad u\in\{0,\tau\}^n,\quad \tau\ge 1.
$$
Hence, the system (\ref{system}) is feasible, which is a contradiction.
\end{proof}

By Theorem \ref{thm2}, if TCP$(\mathcal{A},q)$ is solvable then the corresponding system (\ref{system}) is feasible. More generally, the following theorem gives a necessary and sufficient condition for existence of solutions to
TCP$(\mathcal{A},q)$.
\begin{Theorem}
Given a general TCP$(\mathcal{A},q)$ with $\mathcal{A}\in T_{m,n}$ and $q\in \mathbb{R}^n$, then TCP$(\mathcal{A},q)$ is solvable if and only if the system
\begin{eqnarray}\label{system2}
&& 0\le \mathcal{A}x^{m-1}+q \le \tau e-u, \quad 0\le x\le \tau^{\frac{m-2}{1-m}}u, \nonumber\\
&& u\in\{0,\tau\}^n,\quad \tau\ge 1,
\end{eqnarray}
is feasible.
\end{Theorem}
\begin{proof}
The necessary condition is easily obtained from the proof line of Theorem \ref{thm2}. We now show the sufficient condition. Since the system (\ref{system2}) is feasible, let $(x^*,u^*,\tau^*)$ be a feasible point of (\ref{system2}). Clearly, we have
$$
x^*\ge 0,\quad \mathcal{A}(x^*)^{m-1}+q\ge 0.
$$
Furthermore, when $x^*_i>0$ for some $i\in[n]$, we have $u^*_i=\tau$ and hence $(\mathcal{A}x^{m-1}+q)_i=0$. Thus,
$$
(x^*)^\top\left(\mathcal{A}(x^*)^{m-1}+q\right)=0.
$$
Therefore, $x^*$ is a solution of TCP$(\mathcal{A},q)$. That is, TCP$(\mathcal{A},q)$ is solvable. This completes the proof.
\end{proof}

Let $\mathcal{A}\in S_{m,n}$. By \cite[Theorem 4.5]{cqw2016}, we know that if $\mathcal{A}$ is positive definite then TCP$(\mathcal{A},q)$ is solvable for any $q\in\mathbb{R}^n$ and its solution set is nonempty and compact. Furthermore, the following theorem gives a concrete bound.
\begin{Theorem}
Let $\mathcal{A}\in S_{m,n}$ and $q\in\mathbb{R}^n$. If $\mathcal{A}$ is positive definite, then the solution set of TCP$(\mathcal{A},q)$ is contained in the set
$$
B=\left\{x\in \mathbb{R}^n:\, \|x\|\le \left(\frac{\|q\|}{\lambda_{min}}\right)^{\frac1{m-1}}\right\},
$$
where $\lambda_{min}$ is the smallest Z-eigenvalue of $\mathcal{A}$.
\end{Theorem}
\begin{proof}
Since $\mathcal{A}$ is positive definite, by Theorem 4.5 in \cite{cqw2016}, TCP$(\mathcal{A},q)$ has a nonempty and compact set. Let $S$ be its solution set. In addition, by Theorem 5 in \cite{Qi}, Z-eigenvalues of $\mathcal{A}$ exist and the smallest Z-eigenvalue $\lambda_{min}>0$.  Taking $x\in S$, we have
\begin{equation}\label{bound1}
0=x^\top q+\mathcal{A}x^m.
\end{equation}
Setting $y=x/{\|x\|}$, we have $\|y\|=1$ and $\mathcal{A}x^m=\|x\|^m\mathcal{A}y^m$. It follows from (\ref{smallest}) that
\begin{equation}\label{bound2}
\mathcal{A}x^m\ge \lambda_{min}\|x\|^m.
\end{equation}
Combining (\ref{bound1}) and (\ref{bound2}), we obtain
$$
0\ge \mathcal{A}x^m-\|x\|\|q\|\ge \lambda_{min}\|x\|^m-\|x\|\|q\|=\|x\|\left(\lambda_{min}\|x\|^{m-1}-\|q\|\right),
$$
which yields
$$
\|x\|\le \left(\frac{\|q\|}{\lambda_{min}}\right)^{\frac1{m-1}}.
$$
Hence, $x\in B$.  This completes the proof.
\end{proof}

\section{Numerical experiments}

In this section, we solve MIP (\ref{mip}) to obtain some solutions of some given TCP$(\mathcal{A},q)$ examples. We
consider the following examples for numerical experiments with the mixed integer programming method. We implement all
experiments using Global Solver in LINGO10  on a laptop with an Intel(R) Core(TM) i5-2520M CPU(2.50GHz) and RAM of 4.00GB.

\begin{Example}\label{eg1}
Consider TCP$(\mathcal{A},q)$, where $q=(2,2)^T$ and $\mathcal{A}=(a_{i_1i_2i_3})\in T_{3,2}$ with its entries $a_{111}=1, a_{122}=-1, a_{211}=-2, a_{222}=1$ and other $a_{i_1i_2i_3}=0$.
\end{Example}

By straightforward computation, we obtain that the TCP$(\mathcal{A},q)$ in Example \ref{eg1} has two solutions
$(0,0)^T$ and $(2,\sqrt{6})^T$. Since $q=(2,2)^T>0$, the solution $(0,0)^T$ is trivial by Theorem \ref{thm0}. The corresponding MIP is written as follows:
\begin{eqnarray*}
\max_{\alpha, y, z} && \alpha^2\\
s.t. && 0\leq y_1^2-{y_2^2}+2\alpha^2\leq1-z_1,\\
&& 0\leq -2{y_1^2}+{y_2^2}+2\alpha^2\leq1-z_2,\\
&& 0\leq y_1\leq z_1,\quad 0\leq y_2\leq z_2,\\
&& z_1,z_2\in\{0,1\},\quad \alpha\ge 0.
\end{eqnarray*}

Running the LINGO code, we first obtain a solution of the above MIP:
$$
\alpha^*=0.7071068,\quad y^*=(0,0)^T, \quad z^*=(0,0)^T.
$$
By Theorem \ref{thm1}, we obtain the trivial solution $(0,0)$. We adjust the range of $\alpha$ in the constraints as $0\le \alpha <0.5$ and obtain another solution:
$$
\alpha^*=0.2,\quad y^*=(0.4,0.4898979)^T,\quad z=(1,1)^T.
$$
By Theorem \ref{thm1}, we obtain a solution of the TCP$(\mathcal{A},q)$: $x^*=y^*/\alpha^*=(2,2.4494895)^T\approx (2,\sqrt{6})^T$. We try different range of  $\alpha$ in the constrained region and obtain the two solutions. The details are listed in Table \ref{tab:1}, where {\bf Range} denotes the constraint on $\alpha$ in MIP, {\bf Iter} denotes total solver iterations, {\bf SOL-TCP} denotes the solution of TCP$(\mathcal{A},q)$.

\begin{table}[h]
	\centering
	\caption{The numerical results for Example \ref{eg1}}
	\label{tab:1}       % Give a unique label
	\begin{tabular}{cccccc}
		\hline\noalign{\smallskip}
	Range  & Iter & $\alpha^*$ & $(y^*)^T$ & $(z^*)^T$ & SOL-TCP \\
		\noalign{\smallskip}\hline\noalign{\smallskip}
	$0\le \alpha$ &	39 & 0.7071068 & (0,0) & (0,0)& (0,0)\\
$0\le \alpha\le 0.6$ &29 & 0.6000000 & (0,0) & (0,0)& (0,0)\\
$0\le \alpha\le 0.4$ &        91 & 0.4000000 & (0.8000000,0.9797959)& (1,1) &(2, $\sqrt{6}$) \\
$0\le \alpha\le 0.2$ &       87 & 0.2000000 &(0.4000000,0.4898979) & (1,1) &(2, $\sqrt{6}$) \\
 $0\le \alpha\le 0.1$ &      132 & 0.1000000 &(0.2000000, 0.2449490) &(1,1) &(2, $\sqrt{6}$) \\
		\noalign{\smallskip}\hline
	\end{tabular}
\end{table}

\begin{Example}\label{eg2}
 Consider TCP$(\mathcal{A},q)$, where $q=(-2,-3)^T$ and $\mathcal{A}=(a_{i_1i_2i_3})\in T_{3,2}$ with  $a_{122}=-2, a_{211}=-1$, and other $a_{i_1i_2i_3}=0$.
 \end{Example}

In this exaple, we have
$${\mathcal{A}}x^2+q=\begin{pmatrix}
     -2x_2^2-2 \\
   -x_1^2-3  \\
\end{pmatrix}
<0$$
for any $x\in\mathbb{R}^2$. Clearly, the  TCP$(\mathcal{A},q)$ has no solution. The corresponding MIP is written as follows:
\begin{eqnarray*}
\max_{\alpha, y, z} && \alpha^2\\
s.t. && 0\leq -2{y_2^2}-2\alpha^2\leq1-z_1,\\
&& 0\leq -{y_1^2}-3\alpha^2\leq1-z_2,\\
&& 0\leq y_1\leq z_1,\quad 0\leq y_2\leq z_2,\\
&& z_1,z_2\in\{0,1\},\quad \alpha\ge 0.
\end{eqnarray*}
Running the LINGO code of the above MIP, we obtain the global solution:
$$
\alpha^*=0,\quad y^*=(0.5,0)^T, \quad z^*=(1,0)^T.
$$
By Theorem \ref{thm1}, the  TCP$(\mathcal{A},q)$ in Example \ref{eg2} has no solution. We take different range of $\alpha$ in the above MIP and always obtain the same result. The details are listed in Table \ref{tab:2}.

\begin{table}[h]
	\centering
	\caption{The numerical results for Example \ref{eg2}}
	\label{tab:2}       % Give a unique label
	\begin{tabular}{cccccc}
		\hline\noalign{\smallskip}
	Range  & Iter & $\alpha^*$ & $(y^*)^T$ & $(z^*)^T$ & SOL-TCP \\
		\noalign{\smallskip}\hline\noalign{\smallskip}
	$0\le \alpha$ &	121 & 0 & (0.5,0) & (1,0)& no\\
$0\le \alpha\le 0.6$ &	121 & 0 & (0.5,0) & (1,0)& no\\
$0\le \alpha\le 0.4$ & 	121 & 0 & (0.5,0) & (1,0)& no\\
$0\le \alpha\le 0.2$ &	121 & 0 & (0.5,0) & (1,0)& no\\
 $0\le \alpha\le 0.1$ &	121 & 0 & (0.5,0) & (1,0)& no\\
		\noalign{\smallskip}\hline
	\end{tabular}
\end{table}

\begin{Example}\label{eg3}
 Consider TCP$(\mathcal{A},q)$, where $q=(0,-1)^T$ and ${\mathcal{A}}=(a_{{i_1}{i_2}{i_3}{i_4}})\in T_{4,2}$ with $a_{1111}=1, a_{1112}=-2, a_{1122}=1,a_{2222}=1,$ and  other $a_{{i_1}{i_2}{i_3}{i_4}}=0$.
\end{Example}

This example is taken from \cite{bhw} and its TCP$(\mathcal{A},q)$ has two solutions: $x^*=(0,1)^T$ and $x^*=(1,1)^T$. By simple calculating, we have
$${\mathcal{A}}x^3+q=\begin{pmatrix}
   x_1^3-2x_1^2x_2+x_1x_2^2 \\
   x_2^3-1  \\
\end{pmatrix}
$$
and its corresponding MIP:
\begin{eqnarray*}
\max_{\alpha, y, z} && \alpha^3\\
s.t. && 0\leq y_1^3-2y_1^2y_2+y_1y_2^2\leq1-z_1,\\
&& 0\leq y_2^3-\alpha^3\leq1-z_2,\\
&& 0\leq y_1\leq z_1,\quad 0\leq y_2\leq z_2,\\
&& z_1,z_2\in\{0,1\},\quad \alpha\ge 0.
\end{eqnarray*}
Running the LINGO code of the above MIP, we obtain the global solution:
$$
\alpha^*=1,\quad y^*=(1,1)^T, \quad z^*=(1,1)^T.
$$
By Theorem \ref{thm1}, we can obtain the solution $x^*=(1,1)^T$. We change different range of $\alpha$ in its LINGO code and always obtain the solution, and we don't get another solution $(0,1)$. The details are listed in Table \ref{tab:3}.

\begin{table}[h]
	\centering
	\caption{The numerical results for Example \ref{eg3}}
	\label{tab:3}       % Give a unique label
	\begin{tabular}{cccccc}
		\hline\noalign{\smallskip}
	Range  & Iter & $\alpha^*$ & $(y^*)^T$ & $(z^*)^T$ & SOL-TCP \\
		\noalign{\smallskip}\hline\noalign{\smallskip}
	$0\le \alpha$ &	22 & 1 & (1,1) & (1,1)& (1,1)\\
$0\le \alpha\le 0.8$ &	116 & 0.8 & (0.8001623,0.8) & (1,1)& (1.0002028,1)\\
$0\le \alpha\le 0.6$ &	101 & 0.6 & (0.6001347,0.6) & (1,1)& (1.0002245,1)\\
$0\le \alpha\le 0.4$ &	111 & 0.4 & (0.4005705,0.4) & (1,1)& (1.0014262,1)\\
$0\le \alpha\le 0.2$ &	92 & 0.2 & (0.2009366,0.2) & (1,1)& (1.0046830,1)\\
		\noalign{\smallskip}\hline
	\end{tabular}
\end{table}

We observe the MIP model carefully and find the first two constraints  without the term $\alpha$. Thus, we add a small perturbation $10^5\alpha^3$ and run the LINGO code again. Fortunately, we obtain the following solution of the above MIP:
$$
\alpha^*=1,\quad y^*=(0,1)^T, \quad z^*=(0,1)^T.
$$
By Theorem \ref{thm1}, we can obtain another solution $x^*=(0,1)^T$. The details are listed in Table \ref{tab:4}.

\begin{table}[h]
	\centering
	\caption{The numerical results with small perturbation for Example \ref{eg3}}
	\label{tab:4}       % Give a unique label
	\begin{tabular}{cccccc}
		\hline\noalign{\smallskip}
	Range  & Iter & $\alpha^*$ & $(y^*)^T$ & $(z^*)^T$ & SOL-TCP \\
		\noalign{\smallskip}\hline\noalign{\smallskip}
	$0\le \alpha$ &	29 & 1 & (0,1) & (0,1)& (0,1)\\
$0\le \alpha\le 0.8$ &	152 & 0.8 & (0,0.8) & (0,1)& (0,1)\\
$0\le \alpha\le 0.6$ &	136 & 0.6 & (0,0.6) & (0,1)& (0,1)\\
$0\le \alpha\le 0.4$ &	123 & 0.4 & (0.4005468,0.4) & (1,1)& (1.001367,)\\
$0\le \alpha\le 0.2$ &	96 & 0.2 & (0.2003110,0.2) & (1,1)& (1.001555,1)\\
		\noalign{\smallskip}\hline
	\end{tabular}
\end{table}
 From Tables \ref{tab:3} and \ref{tab:4}, we see that Global Solver in LINGO10 is very sensitive to solve mixed zero-one integer nonlinear optimization problems. The numerical results reported in the above tables show that
  our proposed approach is a good way to solve tensor complementarity problems. It would be much better if there are powerful solvers for mixed integer nonlinear programming problems.

\section{Concluding remarks}
In this paper, we proposed a mixed zero-one integer nonlinear programming (MIP) method to solve tensor complementarity problems (TCPs). For diagonal TCPs, we gave conditions to guarantee the existence of solutions. Specially, we proved that TCPs with diagonal positive definite tensors have a unique solution and then gave a confirmed answer for \cite[Conjecture 5.1]{cqw2016}. We formulated TCPs as MIPs, based on the MIP model, we shown that TCPs are equivalent to mixed integer feasibility problems. Moreover, we gave a sufficient condition for TCPs without solutions, a necessary and sufficient condition for existence of solutions.

We proved that every feasible solution $(\alpha,y,z)$ with $\alpha>0$ of MIP (\ref{mip}) corresponds to a solution of TCP$(\mathcal{A},q)$. Therefore, solving  MIP (\ref{mip}), we may generate several solutions of TCP$(\mathcal{A},q)$. We also reported some numerical results to show the efficiency of the proposed approach.

\end{document}